\documentclass{aims}
\usepackage{amsmath}
  \usepackage{paralist}
  \usepackage{graphics} 
  \usepackage{epsfig} 
\usepackage{graphicx}  \usepackage{epstopdf}
 \usepackage[colorlinks=true]{hyperref}

\usepackage{amsfonts,amssymb}
\usepackage{greekbf}
\usepackage{amsthm}
\usepackage{blkarray}
\usepackage[normalem]{ulem}
\hypersetup{urlcolor=blue, citecolor=red}

\def\currentvolume{X}
 \def\currentissue{X}
  \def\currentyear{200X}
   \def\currentmonth{XX}
    \def\ppages{X--XX}
     \def\DOI{10.3934/xx.xx.xx.xx}

\newtheorem{theorem}{Theorem}[section]

\newtheorem{lemma}[theorem]{Lemma}
\newtheorem{proposition}[theorem]{Proposition}

\theoremstyle{definition}

\newtheorem{remark}{Remark}

\DeclareMathAlphabet{\mathbf}{OT1}{cmr}{bx}{it}
\DeclareMathAlphabet{\mathssb}{OT1}{cmss}{bx}{n}
\DeclareMathAlphabet{\mathssn}{OT1}{cmss}{m}{n}
\DeclareMathAlphabet{\mathub}{OT1}{cmr}{b}{n}
\renewcommand{\d}{\mathbf d}

\newcommand{\e}{{\mathbf e}}

\newcommand{\n}{\mathbf n}

\renewcommand{\u}{\mathbf u}

\newcommand{\A}{\mathbf A}
\newcommand{\B}{\mathbf B}
\newcommand{\C}{\mathbf C}
\newcommand{\D}{\mathbf{D}}

\renewcommand{\L}{\mathbf{L}}

\newcommand{\Ac}{\mathcal A}

\newcommand{\Cc}{\mathcal C}

\newcommand{\Hc}{\mathcal H}
\newcommand{\Ic}{\mathcal I}

\newcommand{\Uc}{\mathcal U}
\newcommand{\Vc}{\mathcal V}
\newcommand{\Sc}{\mathcal S}

\newcommand{\Tc}{\mathcal T}
\newcommand{\Wc}{\mathcal W}
\newcommand{\Xc}{\mathcal X}

\newcommand{\1}{\mathbf{1}}
\newcommand\Real{{\mathord{\rm{I \kern-.22em R}}}}

\newcommand{\beqn}{\begin{equation}}
\newcommand{\eeqn}{\end{equation}}
\newcommand{\bn}{{\bf n}}

\newcommand{\ve}{\varepsilon}

\newcommand{\bu}{{\bf u}}

\newcommand\Tstrut{\rule{0pt}{2.6ex}}         
\newcommand\Bstrut{\rule[-1.3ex]{0pt}{0pt}}   

\title[Nonlocal Curvatures of Surfaces] 
      {On the Nonlocal Curvatures of \\ Surfaces with or without Boundary}

\author[Roberto Paroni, Paolo Podio-Guidugli and Brian Seguin]{}


 \email{paroni@uniss.it}
 \email{ppg@uniroma2.it}
 \email{bseguin@luc.edu}

\thanks{$^*$ Corresponding author: Brian Seguin}

\begin{document}
\maketitle

\centerline{\scshape Roberto Paroni}
\medskip
{\footnotesize
 \centerline{DADU, University of Sassari}
   \centerline{Palazzo del Pou Salit, Piazza Duomo 6, 07041 Alghero (SS), Italy}
} 

\medskip

\centerline{\scshape Paolo Podio-Guidugli}
\medskip
{\footnotesize
 \centerline{Accademia Nazionale dei Lincei}
   \centerline{Palazzo Corsini, Via della Lungara 10, 00165 Roma, Italy}
\medskip

 \centerline{Department of Mathematics, Universita` di Roma TorVergata}
   \centerline{Via della Ricerca Scientifica 1, 00133 Roma, Italy}
}

\medskip

\centerline{\scshape Brian Seguin$^*$}
\medskip
{\footnotesize
 \centerline{Department of Mathematics and Statistics}
   \centerline{Loyola University Chicago, Chicago, IL 60660, USA}
}

\bigskip

 \centerline{(Communicated by the associate editor name)}

\begin{abstract}
For surfaces {\color{black} without boundary}, nonlocal notions of directional and mean curvatures have been recently given. Here, we develop {\color{black}alternative  notions, special cases of which apply to}  surfaces {\color{black} with boundary}. Our main tool is a {\color{black} new fractional or nonlocal} area functional for {\color{black}compact} surfaces. 
\end{abstract}

\section{Introduction}
{{\color{black}In the standard mathematical modelling of thin elastic structures such as plates and shells  a central role is played by the local-curvature fields over their mid surfaces. On learning about the notion of nonlocal mean curvature for surfaces without boundaries, we wondered whether consideration of nonlocal curvatures would allow for capturing certain phenomenologies that are beyond the reach of standard models, so much so when the thin structures under study have a peculiar constitution, such as, say, plates and shells made of polymeric gels. This was our original motivation to try and develop the nonlocal notions of directional and mean curvature we propose in this paper, which are different from those found in the literature and, at variance with them, apply also to surfaces {\color{black} with boundary} embedded in $\Real^n$.}} 
To set the stage, we begin by recalling some  facts. 

Within the framework of the theory of functions with bounded variation on $\Real^n$, the perimeter of a bounded set $E$ with nice boundary $\partial E$  equals the $(n-1)$-dimensional Hausdorff measure of $\partial E$:
\begin{equation}\label{perarea}
\textrm{Per}(E)=\Hc^{n-1}(\partial E)\equiv \textrm{Area}(\partial E)
\end{equation}
(cf.~equation (2.5) in \cite{Alb}). The classical notion of mean curvature is derived from the stationarity condition of the area functional, a crucial step in solving the minimal-surface problem. 
Being local in nature, this notion applies to any smooth surface {\color{black} with or without boundary}.  A relation akin to \eqref{perarea}  between  generalized perimeter and area functionals plays a central role in our paper, because it is from the stationarity of the generalized area functionals we consider that we derive the nonlocal notion of mean curvature we propose.

In recent years, Caffarelli and coworkers \cite{LC,LCMag,CRS,CS,CVb} have motivated the study of $s$-perimeter functionals ($0<s<1/2$), a family of functionals over subsets  of $\Real^n$, whose stationarity condition suggests a definition of nonlocal mean curvature for the closed surface that bounds a candidate minimizer.  The regularity  of  minimizers, called $s$-minimal surfaces, has been investigated by Valdinoci and collaborators \cite{CVa,DiVa,FiVa,SaVa}. Among other things, it is known that $s$-minimal surfaces are smooth off of a singular set of dimension at most $n-8$ for $s$ sufficiently close to $1/2$.  While this is in agreement with a well-known result for classical minimal surfaces,
$s$-minimal surfaces may have  features different from their classical counterparts, in that they may stick to the boundary instead of being transversal to it \cite{DiVa, DiSaVa}.
The motion of surfaces by nonlocal mean curvature has  been investigated using level set methods \cite{CMP1,CMP2,CMP3,Im}.  
For an interesting application, the nonlocal notion of perimeter  has been used to modify the  Gauss free-energy functional used in capillarity theory \cite{MaVa}{\color{black}. A physical motivation for studying this topic is provided by the fact that surfaces with vanishing nonlocal mean curvature arise as limit interfaces of phase-coexistence models with long-tail interactions \cite{SV}}.

The functional delivering the $s$-perimeter of a measurable set $E$ admits the following alternative representations: for $0<s<1/2$, for $\alpha_{n-1}$ the volume of the unit ball in $\Real^{n-1}$, and for $\Cc E$ the complement of $E$ in $\Real^n$,
\beqn\label{sPer}
\begin{aligned}
s\text{-Per}(E)&=\frac{1}{{\alpha_{n-1}}}\int_{\Real^n}\int_{\Real^n}\frac{\chi_{\Cc E}(x)\chi_{E}(y)}{|x-y|^{n+2s}}\, dx dy,
\\&=\int_E\int_{\Cc E}\kappa(x,y) dxdy,\quad \kappa(x,y)=\frac{1}{{\alpha_{n-1}}}\frac{1}{|x-y|^{n+2s}}.
\end{aligned}
\eeqn
As $s\rightarrow 1/2^-$, $s\text{-Per}$ tends, in a sense to be specified later, to the classical perimeter functional studied in \cite{Giusti}.
The first representation in \eqref{sPer} makes it explicit how $s\text{-Per}(E)$ is related to the $H^{s}(\Real^n)$-norm of the characteristic function $\chi_E$ of $E$. The second allows  $s\text{-Per}(E)$ to be interpreted as the evaluation of a distance interaction  between a bounded set $E$ and its complement $\Cc E$ in $\Real^n$, in terms of an integral norm that assigns maximum weight to short-distance pairs $(x,y)$  while keeping track of long two-point  correlations. 

The value at $E$ of  $s\text{-Per}$ is not finite if $E$ is unbounded. In that case, a bounded set $\Omega$ is fixed and the $s$-perimeter of $E$ relative to $\Omega$ is defined in terms of the interaction functional
\begin{equation}\label{inter}
\Ic(A,B):=\int_A\int_{B}\kappa(x,y)\, dx dy,\quad A\cap B =\emptyset,
\end{equation}
in the following way:
\beqn\label{spero}
s\text{-Per}(E,\Omega):=\Ic(E\cap\Omega,\Cc E\cap\Omega)+\Ic(E\cap\Omega,\Cc E\cap\Cc \Omega)+\Ic(E\cap\Cc \Omega,\Cc E\cap\Omega);
\eeqn
 this definition coincides, up to the multiplicative constant $(\alpha_{n-1})^{-1}$, 
with that given in \cite{CRS}.
We recap those properties of $s\text{-Per}$ functionals that are relevant to our present developments in Section 2. 

{\color{black}In this paper,  $\Sc$ denotes  a $(n-1)$-dimensional surface embedded in $\Real^n$, with or without boundary; in the latter instance, we regard $\Sc$ as the complete boundary $\partial E$ of a bounded open set $E$ in $\Real^n$. Our first goal is to develop a notion of $s$-area for whatever $\Sc$. Clearly,  when $\Sc\equiv\partial E$, it would be expedient to have a representation of $s\text{-Per}$ alternative to \eqref{sPer}, according to which the evaluation of the s-perimeter of $E$ depended only on $\partial E$.}
In Section 3 we motivate and discuss the following choice: 
\begin{equation}
 s\text{-Per}(E)=\frac{1}{2}\int_{\Xc(\partial E)}\kappa(x,y)\, dxdy,\nonumber
\end{equation}
where  $\Xc(\partial E)$ consist of all pairs $(x,y)$ such that the straight-line segment $[x,y]$ both has an odd number of cross intersections with $\partial E$ and is not tangent to $\partial E$.
However, {\color{black}to define} the $s$-area of a compact smooth surface $\Sc$ using this formula  with $\partial E$ replaced by $\Sc$ is not viable because the integral on the right in general diverges when $\Sc$ is not the boundary of a set.  Thus, similar to what is done for the $s$-perimeter of an unbounded set, we define the $s$-area of $\Sc$ relative to a chosen bounded set $\Omega$. 
Once we have a definition for $s\text{-area}$ functionals over compact surfaces, we are able to close Section 3 by showing that the $s\text{-area}$ converges, in an appropriate sense, to the classical notion of area, as $s$ approaches $1/2$ {\color{black} from below}. Next, in Section 4, we calculate the first variation of the $s\text{-area}$ functional. The emerging definition of  nonlocal mean curvature, {\color{black}which is meaningful for any surface, compact or otherwise}, is laid down and discussed in Section 5{{\color{black}, where we also adapt to our context the notion of nonlocal directional curvature \cite{AV,PPGL}}. 

{{\color{black} Here is a quick introduction to the new notion of nonlocal mean curvature we propose.
Let  $\Sc$ be an oriented smooth surface.} The 
{\color{black} nonlocal mean curvature at $z\in\Sc$ is 
$$
H_s(z):=\frac{1}{\omega_{n-2}}{PV\!\!\!}\int_{\Real^n}  \widehat\chi_{\Sc}(z,y)|z-y|^{-n-2s}\, dy,
$$
where, for $x\in\Real^n$,
$$
\widehat\chi_\Sc(z,x):=\left\{
\begin{array}{cc}
+1 & \textrm{if}\; x\in \Ac_i(z,1), \\[6pt]
-1 & \textrm{if}\; x\in \Ac_e(z,1).
\end{array}\right.
$$
The sets $\Ac_i(z,1)$ and $\Ac_e(z,1)$ are defined by means of the set $\Xc(\Sc)$, and can be respectively interpreted
as the `interior' and the `exterior' of the surface $\Sc$ relative to the point $z$; while a precise definition is to be found in Section 4, 
Fig.~\ref{AIAEcolor} offers a representation of these sets in a particular case.
\begin{figure}[h]
\centering
\includegraphics[width=4in]{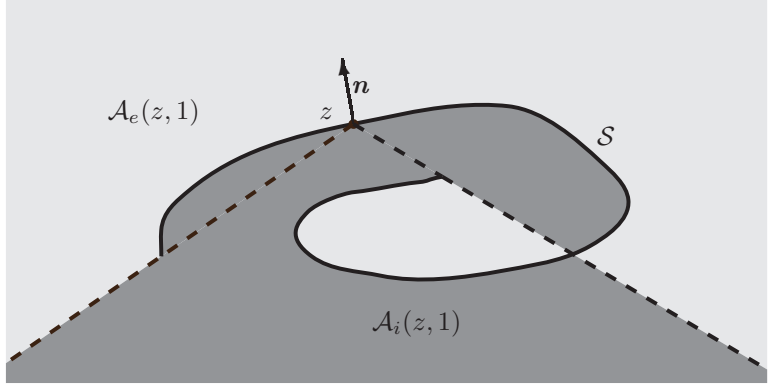}
\thicklines
\put(-170,100){$z$}
\put(-250,100){$\Ac_e(z,1)$}
\put(-150,20){$\Ac_i(z,1)$}
\put(-158,110){$\n$}
\put(-65,90){$\Sc$}
\put(-171.5,105){\rotatebox[origin=c]{100}{$\vector(1,0){25}$}}
\caption{The solid line depicts $\Sc$. The set $\Ac_i(z,1)$ is shown in dark grey,  the set $\Ac_e(z,1)$ in light grey.  
}
\label{AIAEcolor}
\end{figure}
{\color{black}By the use of} a formula of Cabr\'e et al.~\cite{Cetal}, the mean curvature of $\Sc$ can be given the following alternative expression: 
$$
H_s(z)=\frac{1}{s\,\omega_{n-2}} PV\!\!\!\int_{\Sc} |z-y|^{-(n+2s)}(z-y)\cdot \bn_{\Ac_i}(y)\,dy,
$$
where $\bn_{\Ac_i}$ is the outward unit normal to $\Ac_i(z,1)$.
}}
\section{$\boldsymbol s$-perimeter and nonlocal curvatures  of surfaces {\color{black} without boundary}}\label{2}

Let $B_R$ denote the ball of radius $R$ centered at the origin of $\Real^{n}$; throughout the paper we  take {\color{black}$n\ge 2$}. Caffarelli and Valdinoci \cite{CVb} proved that, if for some $R>0$ the set $\partial E\cap B_R$ is $C^{1,\beta}$ for some $\beta\in (0,1)$, then
\beqn\label{perlimit}
\lim_{s\to {1/2}^-} (1-2s)\, s\text{-Per}(E,B_r)=\text{Per}(E,B_r)
\eeqn
for almost every $r\in (0,R)$
(for another result along this line, see  \cite{AdPM}).
The regularity assumption on the boundary of $E$ made in this statement is natural 
for minimizers of the $s$-perimeter functional.  
A set $E\subset \Real^n$ that minimizes $s\text{-Per}(\bar E,\Omega)$ among all the measurable  sets
$\bar E\subset \Real^n$ such that $E\setminus \Omega=\bar E\setminus \Omega$ is called $s$-minimal.
It is proved in \cite{CRS} that, if $E$ is  $s$-minimal,
then $\partial E\cap \Omega$ is of class $C^{1,\beta}$ for some $\beta\in (0,1)$, up to a set of Hausdorff codimension in $\Real^n$ at least equal to $2$.  It is also proved in \cite{CRS} that, if $E$ is an $s$-minimal set in $\Omega$ and $\partial E$ is smooth enough, then $E$ satisfies the Euler-Lagrange equation
of the $s$-perimeter functional: 
$$
H_s=0\quad \textrm{on}\;\partial E.
$$
Here, the nonlocal mean curvature of $E$ at $z\in \partial E$ is defined to be
\begin{equation}\label{defAV}
H_s(z):=\frac{1}{\omega_{n-2}}{PV\!\!\!}\int_{\Real^n}\widetilde\chi_E(y) |z-y|^{-(n+2s)}dy,
\end{equation}
where $PV$ stands for the principal value of the integral,\footnote{In the present instance,
$$
PV\!\!\!\int_{\Real^n}\widetilde\chi_E(y) |z-y|^{-(n+2s)}dy
=\lim_{\varepsilon\rightarrow 0}\int_{\Real^n\setminus B_\varepsilon(z)}\widetilde\chi_E(y) |z-y|^{-(n+2s)}dy.
$$
} $\omega_{n-2}$ is the Hausdorff measure of the $(n-2)$-dimensional unit sphere, and
\begin{equation}\label{tildecar}
\widetilde\chi_E(y):=\left\{\begin{array}{c}\!\!\!\!+1\quad \textrm{if}\; y\in E, \\-1\quad \textrm{if}\; y\in {\mathcal C}E.\end{array}\right.
\end{equation}
This definition of nonlocal mean curvature  coincides with that given in \cite{AV}; the definition given in \cite{Cetal} is the same, to within a multiplicative constant.

Following  \cite{AV}, we now define the nonlocal directional curvature.
Let $y\in\partial E$, let $\e$ be a unit vector tangent to $\partial E$ at $z$, and let
\[\label{pixe}
\pi(z,\e):=\{y\in\Real^n\,|\;y=z+\rho\e+h\n(z),\;\,\rho>0,\;\,h\in\Real\}
\]
be the half-plane through $z$ defined by the unit vector $\e$ and the normal $\n(z)$ (Figure \ref{fig2}); \begin{figure}[h]
\centering%
\includegraphics[height=6cm]{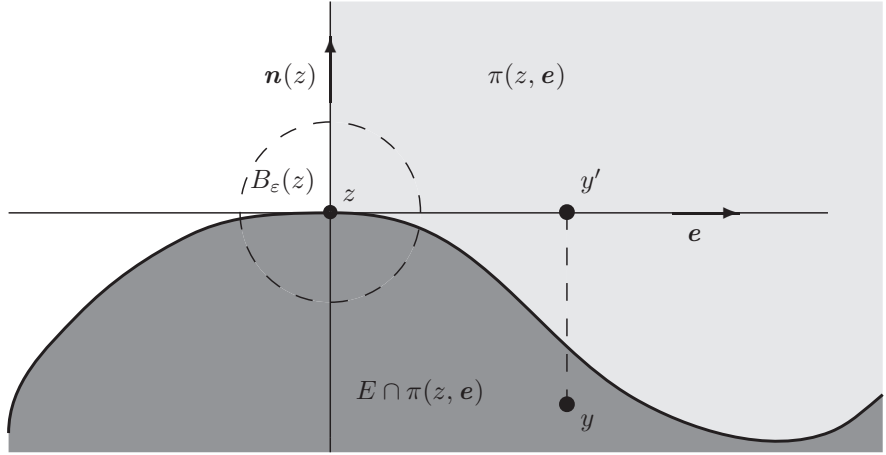}
\thicklines
\put(-200,20){$E\cap \pi(z,\e)$}
\put(-115,10){$y$}
\put(-115,100){$y^\prime$}
\put(-205,95){$z$}
\put(-150,140){$\pi(z,\e)$}
\put(-240,100){$B_\ve(z)$}
\put(-219.5,140){\rotatebox[origin=c]{90}{$\vector(1,0){25}$}}
\put(-235,140){$\n(z)$}
\put(-80,90.5){\rotatebox[origin=c]{0}{$\vector(1,0){25}$}}
\put(-75,80){$\e$}
\caption{{\color{black}The intersection of} an open set $E$ {\color{black}and} the half plane $\pi(z,\e)$.}
\label{fig2}
\end{figure}
moreover, let a superscript prime denote the points of the straight line through $z$ in the direction of $\e$ when they are obtained by projection  in the direction of $\n(z)$ of points of $E\cap \pi(z,\e)$, so that
$ y^\prime=z+\rho \e$.
The nonlocal directional curvature of $E$ at $z$ in the direction $\e$ is
\begin{equation}\label{dirNMC}
K_{s,\e}(z):={PV\!\!\!}\int_{\pi(z,\e)}|y^\prime-z|^{n-2}\,\widetilde\chi_E(y)\,|z-y|^{-(n+2s)} dy,\quad s\in(0,1/2).
\end{equation}

It is proved in \cite{AV} that the nonlocal directional and mean curvatures tend to their local counterparts pointwise in the limit when $s\,\rightarrow\,1/2^-$; precisely, 
\begin{equation}\label{curvlim}
\lim_{s\rightarrow 1/2^-} (1-2s) K_{s,\e}=K_\e,\quad \lim_{s\rightarrow 1/2^-} (1-2s) H_{s}=H.
\end{equation}

\section{The nonlocal area functional}
To motivate the definition of an $s$-area functional related to the $s$-perimeter functional, recall the definition of the $s$-perimeter for a bounded set $E$.  To evaluate the integrals in \eqref{sPer}$_2$, one has to identify pairs of points $x,y\in\Real^n$ such that one point is in $E$ and the other in $\Cc E$. 
Now, we would like to write the $s$-perimeter functional as an integral over a region  depending only on $\partial E$:
\begin{equation}\label{perarearel}
 s\text{-Per}(E)=\frac{1}{2}\int_{\Xc(\partial E)}\kappa(x,y)\, dxdy, \quad \Xc(\partial E)\subset\Real^n\times\Real^n.
\end{equation}
In preparation for choosing such a region, let us consider Figure \ref{figcrossings}.
\begin{figure}[h]
\centering
\includegraphics[width=4in]{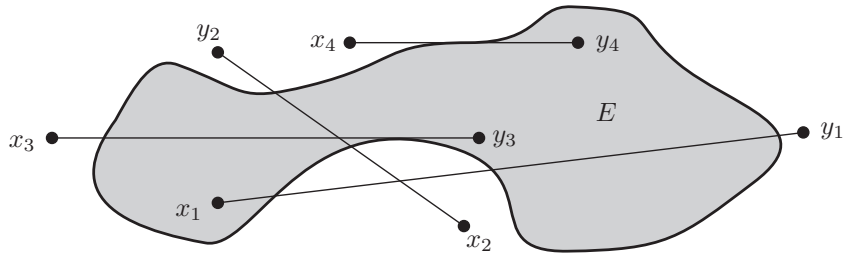}
\thicklines
\put(-240,14){$x_1$}
\put(4,45){$y_1$}
\put(-130,2){$x_2$}
\put(-232,82){$y_2$}
\put(-303,40){$x_3$}
\put(-120,42){$y_3$}
\put(-189,78){$x_4$}
\put(-81,78){$y_4$}
\put(-81,50){$E$}
\caption{Several different ways a straight-line segment can intersect $\partial E$.}
\label{figcrossings}
\end{figure}
We see that   point $x_1$ is internal to $E$,  point $y_1$ external, and the straight-line segment $[x_1,y_1]$, {{\color{black} defined by
$$[x_1,y_1]:=\{ (1-\lambda) x_1+\lambda y_1\ |\ \lambda\in[0,1]\},$$}}  has an odd number of points in common with $\partial E$. We also see that segment $[x_2, y_2]$, which connects two points external to $E$,  has an even number of points in common with $\partial E$. However, the set of pairs with one point in $E$ and the other in $\Cc E$ cannot be characterized by only looking  at the parity of the number of common points the connecting straight-line segment has with $\partial E$. In fact, not all  segments  with  one end in  and the other end out of $E$ have an odd number of points in common with $\partial E$: see e.g.~the  segments $[x_3, y_3]$ and $[x_4,y_4]$ having, respectively, two and infinitely many common points with $\partial E$.

We let  
 $\Xc(\partial E)$ consist of all pairs $(x,y)$ such that the straight-line segment $[x,y]$ both has an odd number of cross intersections with $\partial E$ and is not tangent to $\partial E$.  While it is true that $\Xc(\partial E)\subset  (E\times \Cc E)\cup (\Cc E\times E)$, the previous discussion  shows that the reverse inclusion does not hold. However, as stated in Proposition \ref{intmz} below, the set $\Xc(\partial E)$ differs from $(E\times \Cc E)\cup (\Cc E\times E)$ by a set of $\Hc^{2n}$-measure zero.  This result validates formula \eqref{perarearel}, the main tool to  put together our definition of a nonlocal area functional for a compact surface, {\color{black} with or without boundary}.  To establish Proposition \ref{intmz}, the following change-of-variables formula is needed.

\begin{lemma}\label{lemCOV}
Let $\Sc$ be a compact $(n-1)$-dimensional $C^1$ manifold in $\Real^n$ and let $\Uc_n$ denote the unit sphere in $\Real^n$.  Consider the function 
$$
\Phi:\Sc\times\Uc_n\times\Real^+\times\Real^-\rightarrow \Real^n\times\Real^n
$$ 
defined by
\beqn\label{PhiCOV}
\Phi(z,\u,\xi,\eta):=(z+\xi \u,z+\eta \u)\quad \text{for all}\ (z,\u,\xi,\eta)\in \Sc\times\Uc_n\times\Real^+\times\Real^-,
\eeqn
where $\Real^+$ and $\Real^-$  denote the sets $(0,+\infty)$ and $(-\infty,0)$.
If $\Ac$ is a subset of $\Sc\times\Uc_n\times\Real^+\times\Real^-$ and $f:\Phi(\Ac)\rightarrow\Real$ is any positive integrable function, then
\beqn\label{COV}
\iint_{\Phi(\Ac)}f(x,y)dxdy \le\iiiint_\Ac f(z+\xi\u,z+\eta\u)|\xi-\eta|^{n-1}|\u\cdot\n(z)|dzd\u d\xi d\eta,
\eeqn
where $\n(z)$ is a normal to the surface $\Sc$ at the point $z$. Moreover, if  the restriction of the function $\Phi$ to $\Ac$ is injective, then \eqref{COV} holds with an equality sign.
\end{lemma}

\begin{proof}
See Figure~\ref{figCOV} for a depiction of how $\Phi$ associates $(z,\u,\xi,\eta)$ with points $x$ and $y$ in $\Real^n$.  %
\begin{figure}[h]
\centering
\thicklines
\includegraphics[height=3.5cm]{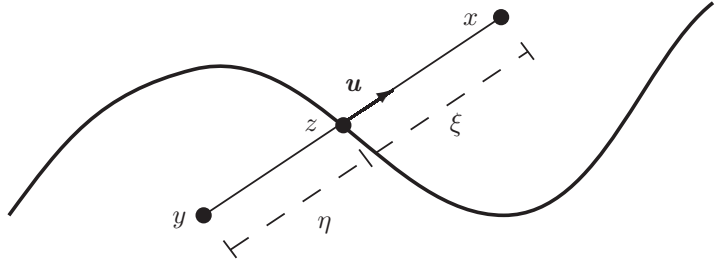}
\put(-155,50){$z$}
\put(-205,15){$y$}
\put(-95,90){$x$}
\put(-150,13){$\eta$}
\put(-100,50){$\xi$}
\put(-148,58){\rotatebox[origin=c]{33.5}{$\vector(1,0){25}$}}
\put(-140,65){$\bu$}
\caption{How the mapping $\Phi$ in \eqref{PhiCOV} associates $(z,\u,\xi,\eta)$ with the pair of points $x$ and $y$ in $\Real^n$.}
\label{figCOV}
\end{figure}

It suffices to prove the Lemma for a set $\Ac=\Sc_\Ac\times \Uc_\Ac\times A_\xi\times A_\eta$
where $\Sc_\Ac\subset \Sc$, $\Uc_\Ac\subset \Uc_n$, $A_\xi\subset\Real^+$, and $A_\eta\subset\Real^-$.
We may further suppose that the sets $\Sc_\Ac$ and $\Uc_\Ac$ can be covered by just one chart (the general case may be reduced to this by means of a partition of unity). This is tantamount to asserting that there are two sets
$P_\Ac, U_\Ac\subset \Real^{n-1}$ and two $C^1$ bijective mappings
$$
P_\Ac\ni p \mapsto \varphi(p)=z\in \Sc_\A
$$
and 
$$
U_\Ac\ni u \mapsto {\boldsymbol\psi}(u)=\u\in \Uc_\A .
$$
For later use, we recall that the integral of a function $g$ over $\Sc_\Ac$ is defined by
\begin{equation}\label{intsurf}
\int_{\Sc_\Ac} g(z)\,dz=\int_{P_\Ac} g(\varphi(p)) J_\varphi\,dp,
\end{equation}
where $J_\varphi=\sqrt{\det \nabla \varphi^T\nabla \varphi}$ is the Jacobian of $\varphi$; the integral over $\Uc_\Ac$ is defined similarly.

Set $A=P_\Ac\times U_\Ac\times A_\xi\times A_\eta$ and define $\tilde \Phi:A\to  \Phi(\Ac)$ by
$$
\tilde \Phi(p,u,\xi,\eta)=\Phi(\varphi(s),{\boldsymbol\psi}(u),\xi,\eta)=(\varphi(s)+\xi {\boldsymbol\psi}(u),\varphi(p)+\eta {\boldsymbol\psi}(u)),
$$
and $\tilde f:A \to \Real$ by 
$$
\tilde f(p,u,\xi,\eta):=f(\varphi(p)+\xi{\boldsymbol\psi}(u),\varphi(p)+\eta{\boldsymbol\psi}(u)).
$$
By the Area Formula (see Theorem 3.9 of Evans and Gariepy \cite{EvGa}),
it follows that 
$$
\iint_{\Phi(\Ac)}\Big[ \sum_{(p,u,\xi,\eta)\in \tilde \Phi^{-1}(x,y)}\tilde f(p,u,\xi,\eta)\Big]dxdy=\iiiint_A\tilde f(p,u,\xi,\eta)|\det \nabla \tilde \Phi|dpdu d\xi d\eta.
$$
{\color{black}Let $\tilde  \Phi^{-1}(x,y)$ be the pre-image through  $\tilde \Phi^{-1}$ of the point $(x,y)$;  s}ince by definition $\tilde f=f \circ \tilde \Phi$,  for any  $(p,u,\xi,\eta)\in \tilde \Phi^{-1}(x,y)$
we have that $\tilde f(p,u,\xi,\eta)=f(x,y)$ and hence
$$
\iint_{\Phi(\Ac)}f(x,y)dxdy \le \iint_{\Phi(\Ac)}\Big[ \sum_{(p,u,\xi,\eta)\in \tilde \Phi^{-1}(x,y)}\tilde f(p,u,\xi,\eta)\Big]dxdy.
$$
Notice that, if the function $\Phi$ restricted to $\Ac$ is injective, then the above equation holds with an equality sign.
Thus
$$
\iint_{\Phi(\Ac)}f(x,y)dxdy \le\iiiint_A f(\varphi(p)+\xi{\boldsymbol\psi}(u),\varphi(s)+\eta{\boldsymbol\psi}(u))|\det \nabla \tilde \Phi|dpdu d\xi d\eta,
$$
from which, taking into account \eqref{intsurf}, the Lemma follows provided that
\begin{equation}\label{Jphi}
|\det \nabla \tilde \Phi|=|\xi-\eta|^{n-1} |\u\cdot \n| J_\varphi J_{\boldsymbol\psi}.
\end{equation}
To prove this identity, we note that the gradient of $\tilde \Phi$ is:
$$
\nabla\tilde\Phi=\begin{blockarray}{ccccc}
n-1 & n-1 & 1 & 1 & \\
    \begin{block}{(c|c|c|c)c}
      \nabla \varphi & \xi \nabla {\boldsymbol\psi} & {\boldsymbol\psi} & \textbf{0}  & n \Bstrut \\ 
      \cline{1-4}
       \nabla \varphi  & \eta \nabla {\boldsymbol\psi} & \textbf{0} & {\boldsymbol\psi} & n \Tstrut \\ 
    \end{block}
\end{blockarray}.
$$
Thus,
\begin{align}
|\det \nabla\tilde \Phi|
& =\Big|\det\begin{blockarray}{cccc}
\phantom{n-1} & \phantom{n-1} & \phantom{1} & \phantom{1}  \\
    \begin{block}{(c|c|c|c)}
      \nabla \varphi & \xi \nabla {\boldsymbol\psi} & {\boldsymbol\psi} &\textbf{0}   \Bstrut \\ 
      \cline{1-4}
       \nabla \varphi & \eta \nabla {\boldsymbol\psi} &\textbf{0} & {\boldsymbol\psi} \Tstrut \\ 
    \end{block}
\end{blockarray}
 \Big|\nonumber\\
 & =\Big|\det\begin{blockarray}{cccc}
\phantom{n-1} & \phantom{n-1} & \phantom{1} & \phantom{1}  \\
    \begin{block}{(c|c|c|c)}
      \nabla \varphi & \xi \nabla {\boldsymbol\psi} & {\boldsymbol\psi} &\textbf{0}   \Bstrut \\ 
      \cline{1-4}
      \textbf{0}& (\eta-\xi) \nabla {\boldsymbol\psi} &-{\boldsymbol\psi} & {\boldsymbol\psi} \Tstrut \\ 
    \end{block}
\end{blockarray}
 \Big|\nonumber\\
& =
 \Big|\det\begin{blockarray}{cccc}
\phantom{n-1} & \phantom{n-1} & \phantom{1} & \phantom{1}  \\
    \begin{block}{(c|c|c|c)}
      \nabla \varphi & \xi \nabla {\boldsymbol\psi} & {\boldsymbol\psi} & \textbf{0}   \Bstrut \\ 
      \cline{1-4}
      \textbf{0} & (\eta-\xi) \nabla {\boldsymbol\psi} &\textbf{0}  & {\boldsymbol\psi} \Tstrut \\ 
    \end{block}
\end{blockarray}
 \Big|\nonumber\\
 & =
 \Big|\det\begin{blockarray}{cccc}
\phantom{n-1} & \phantom{n-1} & \phantom{1} & \phantom{1}  \\
    \begin{block}{(c|c|c|c)}
      \nabla \varphi & {\boldsymbol\psi} &  \xi \nabla {\boldsymbol\psi} & \textbf{0}   \Bstrut \\ 
      \cline{1-4}
      \textbf{0} & \textbf{0} &(\eta-\xi) \nabla {\boldsymbol\psi}  & {\boldsymbol\psi} \Tstrut \\ 
    \end{block}
\end{blockarray}
 \Big|.\nonumber
\end{align}
For upper-block triangular matrices, the following identity holds:
\begin{align*}
\det\begin{blockarray}{cc}
  &  \\
    \begin{block}{(c|c)}
      \A & \B   \\ 
      \cline{1-2}
      {\textbf 0} & \C \Tstrut \\ 
       \end{block}
\end{blockarray}&=\det(\A)\det(\C).
\end{align*}
Consequently,
\begin{align}
\nonumber |\det \nabla\tilde \Phi| &=   \big|\det (\nabla \varphi |{\boldsymbol\psi}) \det ((\eta-\xi) \nabla {\boldsymbol\psi}  | {\boldsymbol\psi})\big| \\
 &= |\eta-\xi|^{n-1}  \big|\det (\nabla \varphi |{\boldsymbol\psi}) \det (\nabla {\boldsymbol\psi}  | {\boldsymbol\psi})\big|;  \label{detdet}
\end{align}
here $(\D|\d)$ denotes the $n\times n$ matrix whose last column is $\d\in \Real^{n}$. 
Let ${\rm cof }\nabla \varphi$ denote the vector whose $i$-th component is equal to $(-1)^{n+i}$ times the 
determinant of the $(n-1)\times(n-1)$ matrix obtained by deleting the $i$-th row from $\nabla \varphi$.
We observe that, for every $i=1,\ldots, n-1$, we have that
$$
0=\det (\nabla \varphi |\frac{\partial \varphi}{\partial p_i})={\rm cof }\nabla \varphi\cdot \frac{\partial \varphi}{\partial p_i},
$$
which implies that ${\rm cof }\nabla \varphi$ is orthogonal to the surface $\Sc$.
By the Cauchy--Binet theorem, it follows that $|{\rm cof }\nabla \varphi|=J_\varphi$; hence,
$$
\n=\frac 1{J_\varphi}{\rm cof }\nabla \varphi
$$
is a unit vector orthogonal to the surface $\Sc$. Similarly we can show that the unit normal ${\boldsymbol \psi}$ to the surface $\Uc_n$ is given by
$$
{\boldsymbol \psi}=\frac 1{J_{\boldsymbol \psi}}{\rm cof }\nabla {\boldsymbol \psi}.
$$
Then
\begin{align*}
\det (\nabla \varphi (p) |{\boldsymbol\psi }(u))&={\rm cof }\nabla \varphi\cdot {\boldsymbol\psi }(u)=J_\varphi \n \cdot \u,\\
\det (\nabla {\boldsymbol\psi} (u)| {\boldsymbol\psi}(u))&=J_{\boldsymbol\psi} {\boldsymbol\psi}(u)\cdot {\boldsymbol\psi}(u)=J_{\boldsymbol\psi},
\end{align*}
and hence from \eqref{detdet} we deduce \eqref{Jphi}.
\end{proof}

Notice that Lemma \ref{lemCOV} holds also for non-orientable surfaces, in which case a  discontinuous normal field may have to be used.
The change-of-variables formula \eqref{COV} is used to prove the next proposition.

\begin{proposition}\label{intmz}
Let $\Sc$ be a compact $(n-1)$-dimensional $C^1$ manifold in $\Real^n$.  The set consisting of all pairs of points $(x,y)\in \Real^n\times\Real^n$ that satisfy at least one of the following conditions:
\begin{enumerate}
\item {\color{black} either} $x\in \Sc$ or $y\in \Sc$;
\item the straight-line segment $[x,y]$ has an infinite number of common points with $\Sc$;
\item the straight-line segment $[x,y]$ is tangent to $\Sc$;
\end{enumerate}
has $\Hc^{2n}$-measure zero.
\end{proposition}

\begin{proof}
Those pairs of points that satisfy Condition 1 are contained in the set $\Sc\times\Real^n\cup \Real^n\times\Sc$, which has $\Hc^{2n}$-measure zero because $\Hc^{n-1}(\Sc)<\infty$.  

Consider now a pair of points $x,y\in\Real^n$ such that Condition  2 holds.  Since both the straight-line segment $[x,y]$ and $\Sc$ are compact, the set of intersection points of $[x,y]$ with $\Sc$ has a cluster point $z$ that is also a point of intersection.  If this cluster point is $x$ or $y$, we are back in the case of Condition 1.  If $z$ is an interior point of $[x,y]$, then either that straight-line segment is tangent to $\Sc$  at that point, and hence we are in the case of Condition 3, or it is not. This latter situation cannot occur, because such points are isolated, in the sense that there is a neighborhood of $z$ in which $[x,y]$  intersects $\Sc$ only once.  {\color{black}This follows from the fact that since $\Sc$ is a $C^1$ surface, there is a neighborhood of $z$ such that $\Sc$ can be approximated by the tangent space $\Tc_z(\Sc)$.  Thus, if a line intersects $\Sc$ at $z$ and is not tangent to $\Sc$ at $z$, then there is a neighborhood of $z$ such that the line only intersects $\Sc$ once in that neighborhood.}  This contradicts the fact that $z$ is a cluster point of intersections.  

To prove the proposition it remains for us to show that the set $\Xc_\text{tan}$ of all pairs of points satisfying Condition  3, has $\Hc^{2n}$-measure zero.  
Let $\Sc_R$ denote the set of all points in $\Real^n$ within a distance $R>0$ from $\Sc$.  Using the function \eqref{PhiCOV}, it follows that
$$
\Xc_\text{tan}\cap(\Sc_R\times\Sc_R)\subset \bigcup_{z\in\Sc}\Phi(\{z\}\times (\Tc_z(\Sc)\cap \Uc_n)\times [0,d_R]\times [-d_R,0]),
$$
where $d_R=\mbox{diameter}(\Sc)+R$  and $\Tc_z(\Sc)$ denotes the tangent space of $\Sc$ at $z$.
Using the change of variables in \eqref{COV}, we have
\begin{align*}
\Hc^{2n}(\Xc_\text{tan})&=\lim_{R\rightarrow \infty}\Hc^{2n}(\Xc_\text{tan}\cap (\Sc_R\times\Sc_R))\\
&=\lim_{R\rightarrow \infty}\int_{\Xc_\text{tan}\cap (\Sc_R\times\Sc_R)}dxdy\\
&\leq\lim_{R\rightarrow \infty}\int_{\Sc} \int_{\Tc_z(\Sc)\cap \Uc_n}\int_{-d_R}^0\int_0^{d_R} |\xi-\eta|^{n-1}|\u\cdot \n(z)| d\xi d\eta d\u dz.
\end{align*}
The last integral is zero  because
$\u \in \Tc_z(\Sc)$.
\end{proof}

No matter if a surface $\Sc$ is the boundary of an open set or not, we can now define $\Xc(\Sc)$  as the set of all pairs of points $(x,y)\in\Real^n\times\Real^n$ such that neither one of the Conditions 2 and 3  in Proposition \ref{intmz} holds and, moreover, the straight-line segment $[x,y]$ has with $\Sc$ an odd number of points in common, not counting its own end points. Proposition \ref{intmz} and the discussion at the beginning of this section guarantee that $\Xc(\partial E)$ and $(E\times\Cc E)\times (\Cc E\times E)$ differ by a set of $\Hc^{2n}$-measure zero, and thus \eqref{perarearel} holds if $E$ has a $C^1$-boundary.

Notice that the right-hand side of \eqref{perarearel} depends on $E$ only through its boundary.  
{\color{black} If we defined}  
 the $s$-area of a compact smooth surface $\Sc$ using the same formula  with $\partial E$ replaced by $\Sc$, then the integral would diverge whenever $\Sc$ is not the boundary of a bounded set.   Thus, similar to what is done for the $s$-perimeter of an unbounded set, we define the $s$-area of $\Sc$ relative to an open and bounded set $\Omega$ by
\beqn\label{sarea}
s\text{-Area}(\Sc,\Omega):=\frac{1}{2}\int_{\Xc(\Sc)}  \kappa(x,y)\max\{\chi_{\Omega}(x),\chi_\Omega(y)\}dxdy.
\eeqn
To see that the $s$-area relative to $\Omega$ is finite, first notice that, since $\kappa(x,y)=\kappa(y,x)$,
\begin{align}
s\text{-Area}(\Sc,\Omega)&=\frac 12 \int_{\Xc(\Sc)}  \kappa(x,y)\chi_{\Omega\times \Omega}(x,y)dxdy+\int_{\Xc(\Sc)}  \kappa(x,y)\chi_{\Omega\times \Cc\Omega}(x,y)dxdy\nonumber\\
&=\frac{1}{2}\int_{\Omega}\int_{\Xc(\Sc,y)\cap\Omega}  \kappa(x,y)dxdy+\int_{\Omega}\int_{\Xc(\Sc,y)\cap\Cc\Omega}\kappa(x,y)dxdy\label{AreasSplit}
\end{align}
where: $\Xc(\Sc,y)=\{x\in \Real^n\ |\ (x,y)\in\Xc(\Sc)\}$. Hence,
\begin{align}
s\text{-Area}(\Sc,\Omega)&
\leq\int_{\Omega}\int_{\Xc(\Sc,y)} \kappa(x,y)dxdy\nonumber\\
&=\int_{\Omega}\int_{\Xc(\Sc,y)\cap \Sc_R}  \kappa(x,y)dxdy+\int_{\Omega}\int_{\Xc(\Sc,y)\setminus \Sc_R} \kappa(x,y)dxdy,\label{sareafinite}
\end{align}
where $\Sc_R$ is the set of all points within a distance $R>0$ from $\Sc$. Now, the integrals on the right in \eqref{sareafinite} turn out to be finite. Indeed, as to the first, choose $R>0$ and large enough so that $\Omega\subset \Sc_R$ and utilize the change of variables in Lemma~\ref{lemCOV} to find that
\begin{align}
\int_{\Omega}\int_{\Xc(\Sc,y)\cap \Sc_R}\frac{1}{|x-y|^{n+2s}}dxdy&\leq \int_{\Sc}\int_{\Uc_n}\int_{-R}^0\int_0^R \frac{|\u\cdot \n(z)|}{|\xi-\eta|^{1+2s}}\, d\xi d\eta d\u dz\nonumber\\
&=\int_\Sc\int_{\Uc_n}\frac{2R^{1-2s}-(2R)^{1-2s}}{2s(1-2s)}|\u\cdot\n(z)|\, d\u dz<\infty;\nonumber
\end{align}
as to the second,  use spherical coordinates centered at $y$ to obtain
\begin{align}
\int_{\Omega}\int_{\Xc(\Sc,y)\setminus \Sc_R}\frac{\omega_{n-1}}{|x-y|^{n+2s}}dxdy\leq \int_\Omega\int_R^\infty\frac{1}{r^{1+2s}}drdy<\infty.\nonumber
\end{align}

Unsurprisingly, the $s$-area of a surface $\Sc$ relative to $\Omega$ satisfies a relation similar to \eqref{perlimit}.  {\color{black} Assume that $\Omega$ is chosen so that $\Sc\subset \Omega$.}

\begin{theorem}\label{convA}
If $\Sc$ is a compact $(n-1)$-dimensional  $C^1$ manifold and $\Omega\subset\Real^n$ is an open and bounded set, then
\beqn\label{slimarea}
\lim_{s\rightarrow 1/2^-} (1-2s)\,s\text{-\rm Area}(\Sc,\Omega)=\text{\rm Area}(\Sc).
\eeqn
\end{theorem}

\begin{proof}
To begin with, set $\ve=\sqrt{1-2s}$, so that as $s$ goes to $1/2$ from the left, $\ve$ goes to zero from the right.  Put
\begin{equation}
\Xc_\ve(\Sc)=\{(x,y)\in \Xc(\Sc)\ |\ |x-y|<\ve \}.\nonumber
\end{equation}
Notice that
\begin{align*}
&\lim_{s\rightarrow 1/2^-}\int_{\Xc(\Sc)\setminus \Xc_\ve(\Sc)}  \frac{1-2s}{|x-y|^{n+2s}}\max\{\chi_{\Omega}(x),\chi_\Omega(y)\}\, dxdy\\
&\hspace{2in}\leq \lim_{s\rightarrow 1/2^-}2\int_\Omega \int_{\{x\in\Real^n |\, |x-y|\geq \ve\}} \frac{1-2s}{|x-y|^{n+2s}}\, dxdy\\
&\hspace{2in}= \lim_{s\rightarrow 1/2^-}2\omega_{n-1}\int_\Omega \int_\ve^\infty \frac{1-2s}{r^{1+2s}}\, drdy\\
&\hspace{2in}=0.
\end{align*}
{\color{black}Since $\max\{\chi_{\Omega}(x),\chi_\Omega(y)\}=1$ for $(x,y)\in \Xc_\ve(\Sc)$ and small $\ve$, it follows that}
\begin{align}
\lim_{s\rightarrow 1/2^-} (1-2s)\,s\text{-Area}(\Sc,\Omega) 
\label{slimarea1}&=\lim_{s\rightarrow 1/2^-} \frac{1}{2\alpha_{n-1}}\int_{\Xc_\ve(\Sc)} \frac{1-2s}{|x-y|^{n+2s}}\, dxdy,
\end{align}
{\color{black}where, as previously defined, $\alpha_{n-1}$ is the volume of the unit ball in $\Real^{n-1}$.}
Even for $s$ close to 1/2, and hence $\ve$ close to $0$, it is possible for the straight-line segment connecting a pair of points $(x,y)\in\Xc_\ve(\Sc)$ to cross $\Sc$ more than once; hence, such a pair $(x,y)$ is not naturally associated with a unique point on the surface $\Sc$.  However, for each pair of points $(x,y)\in\Xc_\ve(\Sc)$ we can arbitrarily choose a point $z\in \Sc$ that lies on the straight-line segment joining $x$ and $y$.  Denote this point by $c(x,y)$.  One can think of $c$ as a function from $\Xc_\ve(\Sc)$ to $\Sc$ that singles out a particular crossing for the pair $(x,y)\in \Xc_\ve(\Sc)$.  There are many such functions, here we choose one.  For each $z\in\Sc$ and for each  $\bu\in\Uc_n$, set
\begin{align}
C_\ve(z,\u):=\{(\xi,\eta)\in\Real^+\times\Real^-\ | &\ (z+\xi \u,z+\eta \u)\in\Xc_\ve(\Sc)\ \nonumber\\
&\qquad\quad\text{and}\ c(z+\xi \u,z+\eta \u)=z \}.\label{Cve}
\end{align}
Notice that the function $\Phi$ defined in Lemma \ref{lemCOV} is injective on the set
$$\bigcup_{(z,\bu)\in \Sc\times\Uc_n} \{z\}\times\{\u\} \times C_\ve(z,\u).$$
To see this, consider two quadruplets $(z_1,\u_1,\xi_1,\eta_1)$ and $(z_2,\u_2,\xi_2,\eta_2)$  in this set  that 
$$\Phi(z_1,\u_1,\xi_1,\eta_1)=\Phi(z_2,\u_2,\xi_2,\eta_2)=(x,y)\in \Xc_\ve(\Sc).$$
We know that the straight-line segment $[x,y]$ crosses $\Sc$ an odd number of times.  Since both quadruplets gets mapped to $(x,y)$   and $(\xi_1,\eta_1)$ and $(\xi_2,\eta_2)$ belong to $C_\ve(z,\u)$, we must have that
$$c(z_1+\xi_1 \u_1 , z_1 + \eta_1 \u_1)=z=c(z_2+\xi_2 \u_2 , z_2 + \eta_2 \u_2).$$  
It then follows from \eqref{Cve} that $z=z_1=z_2$.  Moreover, since $\xi_1$ and $\xi_2$ are positive,  the equality chain 
$$x=z+\xi_1\u_1=z+\xi_2 \u_2$$ 
holds only for $\xi_1=\xi_2$ and $\u_1=\u_2$;  similarly, it follows that $\eta_1=\eta_2$.  Thus, $\Phi$ is injective.

Now, using the change of variables \eqref{COV} with equality sign, we have that {\color{black}
\begin{align}
\frac{1}{2}\int_{\Xc_\ve(\Sc)}\frac{1-2s}{|x-y|^{n+2s}}\, dxdy &\,  = \, \frac{1}{2}\int_{\Sc}\int_{\Uc_n} \int_{C_\ve(z,\u)} \frac{1-2s}{|\xi-\eta|^{1+2s}} |\u\cdot \n(z)| d\xi d\eta d\u dz\nonumber\\
&\hspace{-1in}= \, \frac{1}{2}\int_{\Sc}\int_{\{\u\in\Uc_n\, |\, \u\cdot \n(z)>0\}} \int_{C_\ve(z,\u)} \frac{1-2s}{|\xi-\eta|^{1+2s}} \u\cdot \n(z) d\xi d\eta d\u dz\nonumber\\
&\hspace{-0.9in}\,- \, \frac{1}{2}\int_{\Sc}\int_{\{\u\in\Uc_n\, |\, \u\cdot \n(z)<0\}} \int_{C_\ve(z,\u)} \frac{1-2s}{|\xi-\eta|^{1+2s}} \u\cdot \n(z) d\xi d\eta d\u dz.\label{unminus}
\end{align}}
For each $z\in\Sc$ and $\u\in\Uc_n$ such that $\u\cdot\n(z)>0$ there exists an $\ve_0>0$
such that the segment $\{z+s\u\ |\  |s|\le \ve_0\}$ intersects the surface $\Sc$ only at $z$. Hence, for 
$\ve\le \ve_0$ we have that
$$
C_\ve(z,\u)=\{(\xi,\eta)\in\Real^+\times\Real^-\ |\ \xi-\eta\leq \ve \},
$$
so that
\begin{align*}
\lim_{s\rightarrow 1/2^-} \int_{C_\ve(z,\u)} \frac{1-2s}{|\xi-\eta|^{1+2s}} d\xi d\eta&= \lim_{s\rightarrow 1/2^-} \int_{-\ve}^0\int_0^{\ve+\eta} \frac{1-2s}{|\xi-\eta|^{1+2s}} d\xi d\eta \\
&= \lim_{s\rightarrow 1/2^-} \frac{\ve^{1-2s}}{2s}  
= 1.
\end{align*}
{\color{black}The same result can be reached for  $z\in\Sc$ and $\u\in\Uc_n$ such that $\u\cdot\n(z)<0$.}
These facts together with the dominated convergence theorem allows for the calculation of the limit of \eqref{unminus}.  Namely, if $B^{n-1}$ is the unit ball in $\Real^{n-1}$, then
\begin{align*}
\lim_{s\rightarrow 1/2^-}\frac{1}{2\alpha_{n-1}}\int_{\Xc_\ve(\Sc)}&\frac{1-2s}{|x-y|^{n+2s}}\, dxdy \\ & ={\color{black}\frac{1}{2\alpha_{n-1}}\int_{\Sc}\int_{\{\u\in\Uc_n\, |\, \u\cdot \n(z)>0\}} \u\cdot \n(z)d\u dz}\\
&\hspace{5mm} {\color{black}-\frac{1}{2\alpha_{n-1}}\int_{\Sc}\int_{\{\u\in\Uc_n\, |\, \u\cdot \n(z)<0\}} \u\cdot \n(z)d\u dz}\\
& =\frac{1}{\alpha_{n-1}}\int_{\Sc}\int_{\{\u\in\Uc_n\, |\, \u\cdot \n(z)>0\}} \u\cdot \n(z)d\u dz\\
&=\frac{1}{\alpha_{n-1}}\int_{\Sc}\int_{B^{n-1}}\int_{0}^{\pi/2} \sin(\theta)d\theta dA dz\\
&= \text{Area}(\Sc).
\end{align*}
Putting this together with \eqref{slimarea1} shows that \eqref{slimarea} holds.

\end{proof}

\section{The first variation of the $s$-area functional}

Motivated by the connection between the local mean curvature of a surface and the first variation of the area functional,  we calculate the first variation of the $s$\text{-area} functional.

From now on in this section we restrict attention to compact $(n-1)$-dimensional $C^1$ manifold in $\Real^n$, which we choose to be orientable. Let $\Sc$ be such a surface with $\bn$ its chosen normal field,  let $\Omega$ be an open bounded set that contains $\Sc$, and let 
 $\phi:\Sc\rightarrow\Real$ be a normal variation of $\Sc$, that is, a continuously differentiable function that is zero on $\partial\Sc$.  
For  $\ve>0$, define
$$\Sc_\ve:=\{z+\ve\phi(z)\bn(z)\ |\ z\in\Sc\};$$
note that $\partial \Sc_\ve=\partial\Sc$ and that $\Sc_\ve$ is
a compact $(n-1)$-dimensional manifold for small $\ve$.
We wish to find a characterization of those $\Sc$ that satisfy 
\beqn\label{varlim}
\lim_{\epsilon\rightarrow 0^+} \frac{s\text{-Area}(\Sc_\ve,\Omega)-s\text{-Area}(\Sc,\Omega)}{\ve}=0
\eeqn
 for all normal variations $\phi$.

\begin{figure}[h]
\centering
\includegraphics[width=4in]{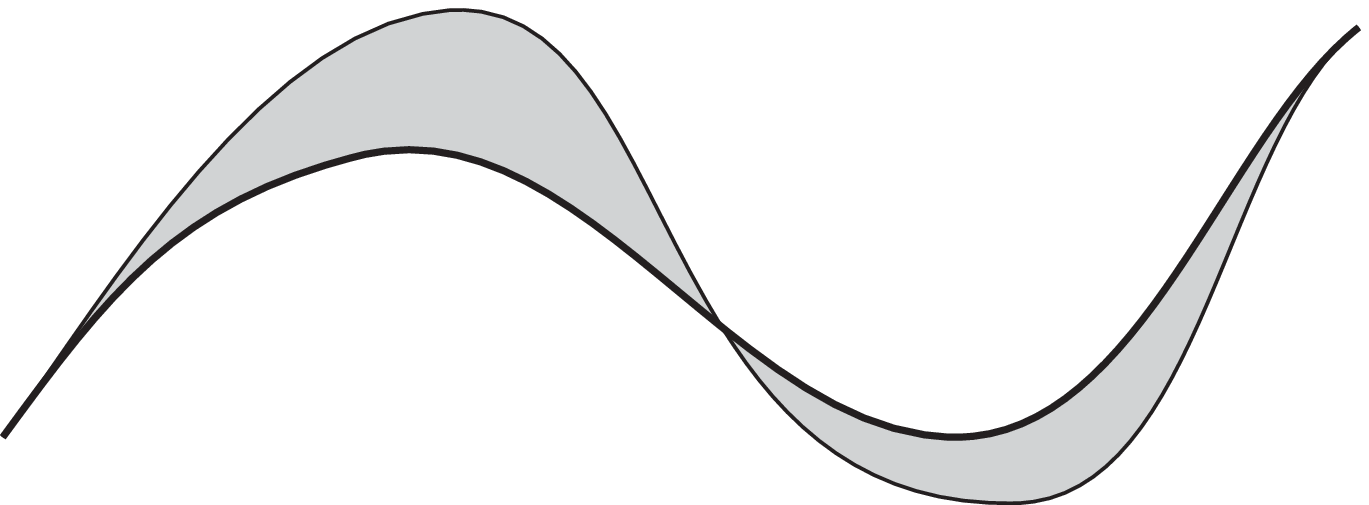}
\thicklines
\put(-40,15){$\Sc_\ve$}
\put(-50,50){$\Sc$}
\put(-220,80){$\Vc_\ve$}
\put(-200,85){\rotatebox[origin=c]{78}{$\vector(1,0){30}$}}
\put(-180,105){$z+\ve \phi(z)\bn(z), \  \phi(z)>0$}
\put(-186,102){$\bullet$}
\put(-193,72){$\bullet$}
\put(-193,65){$z$}
\put(-90,21){$\bar z$}
\put(-89.5,12.5){$\bullet$}
\put(-87.5,3){\rotatebox[origin=c]{270}{$\vector(1,0){14}$}}
\put(-89.5,-2){$\bullet$}
\put(-110,-13){$\bar z+\ve \phi(\bar z)\bn(\bar z), \ \phi(\bar z)<0$}
\caption{A depiction of $\Sc$, $\Sc_\ve$, and $\Vc_\ve$.}
\label{Vepsilon}
\end{figure}

Let $\Vc_\ve$ denote the region inclosed by the surfaces $\Sc_\ve$ and $\Sc$, {\color{black} so that
$$ \Vc_\ve:=\{ z+\zeta \bn\ |\ z\in\Sc,\, \phi(z)\not=0,\  0< \zeta/\phi(z)< \ve \}$$}
(see Figure~\ref{Vepsilon}), and let
$$\Wc_\ve:=(\Vc_\ve\times\Cc\Vc_\ve)\cup(\Cc\Vc_\ve\times\Vc_\ve).$$ 
In view of Proposition \ref{intmz}, for $\Hc^{2n}$-almost every pair $(x,y)\in\Wc_\ve$, the number of points the straight-line segment connecting $x$ and $y$ has in common with $\Sc$ (not counting its  end points) differs by {\color{black} an odd number} from the number of points it has in common with $\Sc_\ve$. 
{\color{black}
Indeed, for $\Hc^{2n}$-almost every pair $(x,y)\in\Wc_\ve$ the segment $[x,y]$ intersects $\partial \Vc_\ve$ an odd number of times; let us denote by $\#_{\partial \Vc_\ve}$ this odd number. Let $\#_{\Sc}$ and $\#_{\Sc_\ve}$ denote the number of times that $[x,y]$ intersect $\Sc$ and $\Sc_\ve$, respectively. Since $\partial \Vc_\ve=\Sc\cup\Sc_\ve$ we have that  $\#_{\partial \Vc_\ve}=\#_{\Sc}+\#_{\Sc_\ve}$. But the parity of $\#_{\Sc}-\#_{\Sc_\ve}$ coincides with the parity of $\#_{\Sc}+\#_{\Sc_\ve}$ and hence it is odd just like $\#_{\partial \Vc_\ve}$.
}

 Thus, up to a set of $\Hc^{2n}$-measure zero,
\beqn\label{Xcep}
\Xc(\Sc_\ve)=\big(\Cc\Xc(\Sc)\cap\Wc_\ve\big)\cup\big(\Xc(\Sc)\backslash\Wc_\ve\big).
\eeqn
{\color{black}
With 
$$
D_\eta:=\{(x,y)| \, |x-y|>\eta\},
$$
for $\eta>0$, i}t follows that, for 
$$f(x,y):=|x-y|^{-n-2s}\max\{\chi_{\Omega}(x),\chi_{\Omega}(y)\}{\color{black}\chi_{D_\eta}(x,y)}=f(y,x),$$ 
we have that
\begin{align*}
&\Big(\int_{\Xc(\Sc_\ve)} - \int_{\Xc(\Sc)}\Big)f(x,y)\, dxdy\\
&\hspace{1in}= \Big(\int _{\Cc\Xc(\Sc)\cap\Wc_\ve}+\int_{\Xc(\Sc)\backslash\Wc_\ve} - \int_{\Xc(\Sc)\backslash\Wc_\ve} - \int_{\Xc(\Sc)\cap\Wc_\ve}\Big)f(x,y)\, dxdy\\
&\hspace{1in}=  \Big(\int _{\Cc\Xc(\Sc)\cap\Wc_\ve}- \int _{\Xc(\Sc)\cap\Wc_\ve}\Big) f(x,y)\, dxdy.
\end{align*}
Hence, \eqref{varlim} is equivalent to the condition
\beqn\label{varlim2}
\lim_{\ve\rightarrow 0^+}\frac{1}{\ve} \Big( \int_{\Cc\Xc(\Sc)\cap\Wc_\ve} f(x,y)\, dxdy - \int_{\Xc(\Sc)\cap\Wc_\ve} f(x,y)\, dxdy\Big)=0.
\eeqn
Moreover, as $f(x,y)=f(y,x)$,
\begin{align*}
\int_{\Cc\Xc(\Sc)\cap\Wc_\ve} f(x,y)\, dxdy & = \Big(\int_{\Cc\Xc(\Sc)\cap(\Vc_\ve\times\Cc\Vc_\ve)} + \int_{\Cc\Xc(\Sc)\cap(\Cc\Vc_\ve\times\Vc_\ve)}\Big)f(x,y)\, dxdy\\
&=2\int_{\Cc\Xc(\Sc)\cap(\Vc_\ve\times\Cc\Vc_\ve)} f(x,y)\, dxdy\\
&= 2\int_{\Vc_\ve}\int_{\{y\in\Cc\Vc_\ve\, |\, (x,y)\in\Cc\Xc(\Sc)\}}f(x,y)\, dydx;
\end{align*}
similarly,
$$
\int_{\Xc(\Sc)\cap\Wc_\ve} f(x,y)\, dxdy = 2\int_{\Vc_\ve}\int_{\{y\in\Cc\Vc_\ve\, |\, (x,y)\in\Xc(\Sc)\}}f(x,y)\, dydx.
$$

Now, for all $x\in\Vc_\ve$ define
\begin{align*}
f^{\Cc\Xc}_\ve(x)&:=\int_{\{y\in\Cc\Vc_\ve\, |\, (x,y)\in\Cc\Xc(\Sc)\}}f(x,y)\, dy,\\
f^{\Xc}_\ve(x)&:=\int_{\{y\in\Cc\Vc_\ve\, |\, (x,y)\in\Xc(\Sc)\}}f(x,y)\, dy,
\end{align*}
and, for $z\in\Sc$, define
\begin{align*}
\Ac_e(z,\phi)&:=\big\{y\in \Real^n\ |\ \big((z,y)\in \Xc(\Sc)\ \text{and}\ \phi(z)(z-y)\cdot \n(z)> 0\big)\\
&\hspace{1in}\text{or } \big((z,y)\in \Cc\Xc(\Sc)\ \text{and}\ \phi(z)(z-y)\cdot \n(z)< 0\big)\big\},\\
\Ac_i(z,\phi)&:=\big\{y\in \Real^n\ |\ \big((z,y)\in \Cc\Xc(\Sc)\ \text{and}\ \phi(z)(z-y)\cdot \n(z)> 0\big)\\
&\hspace{1in}\text{or } \big((z,y)\in \Xc(\Sc)\ \text{and}\ \phi(z)(z-y)\cdot \n(z)< 0\big)\big\},
\end{align*}
(see Figures~\ref{AIAE} and \ref{AIAEcolor}).  {\color{black}When $\phi(z)\not=0$,} the sets $\Ac_e(z,\phi)$ and $\Ac_i(z,\phi)$ complement each other up to a set of $\Hc^{n}$-measure zero.  {\color{black} Moreover, using Proposition~3.62 in \cite{AFP}, it can be shown that $\Ac_e(z,\phi)$ and $\Ac_i(z,\phi)$ locally have finite perimeter and so have an exterior unit normal on their reduced boundary.} In particular, 
\begin{figure}[h]
\centering
\includegraphics[width=4in]{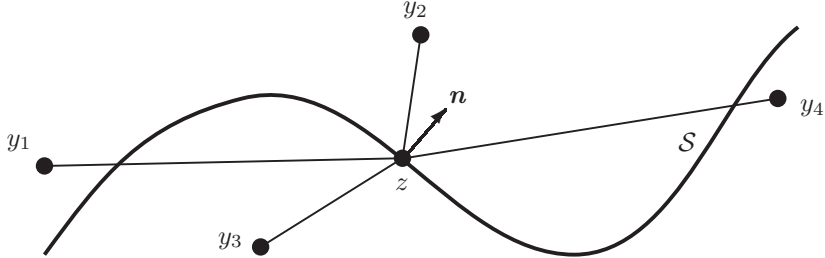}
\thicklines
\put(-158,45){\rotatebox[origin=c]{49}{$\vector(1,0){25}$}}
\put(-46,40){$\Sc$}
\put(-153,25){$z$}
\put(-133,58){$\n$}
\put(-300,42){$y_1$}
\put(-150,92){$y_2$}
\put(-220,5){$y_3$}
\put(0,55){$y_4$}
\caption{Here, $\phi(z)>0$; $y_1, y_2 \in \Ac_e(z,\phi)$ and $y_3,y_4\in \Ac_i(z,\phi)$.}
\label{AIAE}
\end{figure}
if $\Sc$ is the boundary of a set $E$, $\bn$ is the exterior normal, and $\phi>0$, then, up to a set of $\Hc^n$-measure zero, $\Ac_e(z,\phi)$ consists of the points outside of $E$ and $\Ac_i(z,\phi)$ consists of the points in $E$.

We state a useful generalization of a result by Weyl \cite{Weyl} {\color{black}(third to last formula on page 464)}, {\color{black} which can be obtained by use of the area formula}: given an integrable function $g$ defined on $\Vc_\ve$, we have
\beqn\label{neededint}
\int_{\Vc_\ve}g(x)\, dx = \int_{\Sc}\int_{-\ve\phi^-(z)}^{+\ve\phi^+(z)} g(z+\xi\bn(z)) \text{det}_{\Tc_z(\Sc)}(\1-\xi\L(z))\, d\xi dz
\eeqn
where $\phi^+$ and $\phi^-$ are the positive and negative parts of $\phi$, $\text{det}_{\Tc_z(\Sc)}$ is the determinant function for linear mappings from $\Tc_z(\Sc)$, the tangent space of $\Sc$ at $z$, into itself, and $\L$ is the curvature tensor for $\Sc$, {\color{black}which is defined as $-\nabla_\Sc \bn$, the surface gradient of the normal vector field.}

With the use of \eqref{neededint}, the limit in \eqref{varlim2} can be computed:
\begin{align*}
\lim_{\ve\rightarrow 0^+} \frac{1}{\ve} \int_{\Cc\Xc(\Sc)\cap\Wc_\ve} f(x,y)\, dxdy &= \lim_{\ve\rightarrow 0^+} \frac{2}{\ve}\int_{\Vc_\ve} f^{\Cc\Xc}_\ve(x)\, dx\\
&\hspace{-1.5in}= 2\lim_{\ve\rightarrow 0^+}\int_{\Sc} \frac{1}{\ve}\int_{-\ve\phi^-(z)}^{+\ve\phi^+(z)} f^{\Cc\Xc}_\ve(z+\xi\bn(z)) \text{det}_{\Tc_z(\Sc)}(\1-\xi\L(z))\, d\xi dz\\
&\hspace{-1.5in}= 2\lim_{\ve\rightarrow 0^+}\int_{\Sc} \int_{-\phi^-(z)}^{+\phi^+(z)} f^{\Cc\Xc}_\ve(z+\ve\zeta\bn(z)) \text{det}_{\Tc_z(\Sc)}(\1-\ve \zeta\L(z))\, d\zeta dz.
\end{align*}
{\color{black} For every $z\in \Sc$ and $\zeta\in \Real$ such that $\phi(z)\zeta>0$, let
$$
E_\ve:=\{y\in\Cc\Vc_\ve\, |\, (z+\ve\zeta\bn(z),y)\in\Cc\Xc(\Sc)\}.
$$ 
Then, $\chi_{E_\ve}\to \chi_{\Ac_e(z,\phi)}$ in $L^1(\Real^n)$
and hence, for almost every $z$,
$$
\lim_{\ve\to 0}f^{\Cc\Xc}_\ve(z+\ve\zeta\bn(z))= \int_{\Ac_e(z,\phi)} f(z,y)\,dy.
$$
Since 
$$
|f(x,y)|\le \frac{1}{\eta^{n+2s}}\max\{\chi_{\Omega}(x),\chi_{\Omega}(y)\},
$$
by the dominated convergence theorem, w}e conclude that
\begin{align*}
\lim_{\ve\rightarrow 0^+} \frac{1}{\ve} \int_{\Cc\Xc(\Sc)\cap\Wc_\ve} f(x,y)\, dxdy &=
 2\int_{\Sc} \int_{-\phi^-(z)}^{+\phi^+(z)} \int_{\Ac_e(z,\phi)} f(z,y)\,dy \, d\zeta dz\\
&\hspace{-1.5in} 
 = 2\int_\Sc (\phi^+(z)+ \phi^-(z)) \int_{\Ac_e(z,\phi)} f(z,y)\,dy dz\\
&\hspace{-1.5in} = 2\int_\Sc  |\phi(z)| \int_{\Ac_e(z,\phi)}f(z,y)\, dydz.
\end{align*}
Similarly,
$$
\lim_{\ve\rightarrow 0^+} \frac{1}{\ve} \int_{\Cc\Xc(\Sc)\cap\Wc_\ve} f(x,y)\, dxdy= 2\int_\Sc  |\phi(z)| \int_{\Ac_i(z,\phi)}f(z,y)\, dydz.
$$
Thus, \eqref{varlim2} is equivalent to
\beqn\label{varlim3}
\int_\Sc  |\phi(z)| \int_{\Ac_e(z,\phi)}f(z,y)\, dydz = \int_\Sc  |\phi(z)| \int_{\Ac_i(z,\phi)}f(z,y)\, dydz,
\eeqn
a condition which holds whatever $\phi$ if and only if, for all $z\in\Sc$,
\beqn\label{varlim4}
\int_{\Ac_e(z,1)}f(z,y)\, dy = \int_{\Ac_i(z,1)}f(z,y)\, dy,
\eeqn
{\color{black} where $\Ac_i(z,1):=\Ac_i(z,\phi(\cdot)=1)$ and $\Ac_e(z,1):=\Ac_e(z,\phi(\cdot)=1).$} Equation \eqref{varlim4} is found by fixing a $z\in\Sc$ and considering variations induced by a positive-valued $\phi$,  whose support is contained in a small neighborhood of $z$.  Considering variations associated with negative-valued $\phi$ would lead to the same condition, because
$$
\Ac_e(z,-1)=\Ac_i(z,1)\ \text{and}\ \Ac_i(z,-1)=\Ac_e(z,1).
$$
{\color{black} Recalling the definition of $f$ and that $z\in \Sc\subset \Omega$,  \eqref{varlim4} writes as
$$
\int_{\Ac_e(z,1)\setminus B_\eta(z)}|z-y|^{-n-2s}\, dy = \int_{\Ac_i(z,1)\setminus B_\eta(z)}|z-y|^{-n-2s}\, dy,
$$
and letting $\eta$ go to zero we find the  following  result.}
\begin{theorem}
	Let $\Sc$ be an orientable compact $(n-1)$-dimensional $C^1$ manifold in $\Real^n$ and let $\Omega$ be an open bounded set that contains $\Sc$.  A necessary and sufficient condition for the vanishing  of the first variation  with respect to surfaces with the same boundary   of the $s$-area {\color{black} of $\Sc$} relative to $\Omega$ is 
\beqn\label{fv}
{\color{black}PV\!\!\!}\int_{\Ac_e(z,1)}|z-y|^{-n-2s}\, dy = {\color{black}PV\!\!\!}\int_{\Ac_i(z,1)}|z-y|^{-n-2s}\, dy,
\eeqn
for each $z\in\Sc$.
\end{theorem}

\section{Nonlocal curvatures of surfaces}

Motivated by condition \eqref{fv} for the vanishing of the first variation of the $s$-area relative to $\Omega$, we define  as follows the nonlocal mean curvature of an orientable $C^1$ surface $\Sc$, {\color{black} which need not be compact}, at its point $z$:
\beqn\label{Hs}
H_s(z):=\frac{1}{\omega_{n-2}}{PV\!\!\!}\int_{\Real^n}  \widehat\chi_{\Sc}(z,y)|z-y|^{-n-2s}\, dy,
\eeqn
where, for $x\in\Real^n$,
\begin{equation}\label{hatcar}
\widehat\chi_\Sc(z,x):=\left\{
\begin{array}{cc}
+1 & \textrm{if}\; x\in \Ac_i(z,1), \\[6pt]
-1 & \textrm{if}\; x\in \Ac_e(z,1).

\end{array}\right.
\end{equation}
Notice that $H_s(z)$ does not depend on the choice of $\Omega$.
When the surface $\Sc$ is the boundary of an open set,  {\color{black}formulas \eqref{Hs}-\eqref{hatcar} are consistent with  formulas \eqref{defAV}-\eqref{tildecar}} holding for surfaces {\color{black} without boundary}. 

Cabr\'e et al.~\cite{Cetal} noticed that
\beqn\label{divID}
|z-y|^{-(n+2s)}=\frac{1}{2s}\text{div}_y\big[|z-y|^{-(n+2s)}(z-y)\big],
\eeqn
which, together with the divergence theorem, allows the {\color{black}nonlocal-mean-curvature formula \eqref{defAV}} for a  surface {\color{black} without boundary} to be written as
$$
H_s(z)=\frac{1}{s\,\omega_{n-2}} PV\!\!\!\int_{\partial E} |z-y|^{-(n+2s)}(z-y)\cdot \bn(y)\,dy.
$$
{\color{black}We now show that an analogous result  also holds for formula \eqref{Hs}.  

\color{black}
\begin{proposition}\label{Hrep}
Let $\Sc$ be an oriented $C^1$ surface. For $z\in\Sc$,  let $\n_{\Ac_i}$ be the outward unit normal to $\Ac_i(z,1)$.  The nonlocal mean curvature of $\Sc$ at $z$ satisfies
\beqn\label{33}
H_s(z)=\frac{1}{s\,\omega_{n-2}} PV\!\!\!\int_{\Sc} |z-y|^{-(n+2s)}(z-y)\cdot \bn_{\Ac_i}(y)\,dy.
\eeqn
\end{proposition}

\begin{proof}
We start by noticing that $\Sc$ is contained in the boundary $\partial\Ac_i(z,1)$ of $\Ac_i(z,1)$, and that
if
\beqn\label{yn0}
y\in \partial\Ac_i(z,1)\setminus (\Sc\cup\partial\Sc)   \Longrightarrow (z-y)\cdot \n_{\Ac_i}(y)=0.
\eeqn
Indeed, assume that there is a {\color{black}point} $y\in \partial\Ac_i(z,1)\setminus (\Sc\cup\partial\Sc)$ for which $(z-y)\cdot \n_{\Ac_i}(y)\ne0$.
Then, there exists an $\varepsilon>0$ such that {\color{black}$B(y,\varepsilon)$, the ball  of  radius $\varepsilon$ centered at $y$,} does not intersect the surface $\Sc$. Also, from $(z-y)\cdot \n_{\Ac_i}(y)\ne0$
we deduce that there are two points $y^i,y^e\in B(y,\varepsilon)$ that lie on the straight line passing trough the points $z$ and $y$ and such that $y^i\in \Ac_i(z,1)$ and $y^e\notin \Ac_i(z,1)$. 
Since $B(y,\varepsilon)\cap\Sc=\emptyset$, either both $(z,y^i)$ and $(z,y^e)$ belong to $\Xc(\Sc)$ or {\color{black}they both belong} to $\Cc\Xc(\Sc)$. But since
$$
(z-y^i)\cdot \n(z)=(z-y^e)\cdot \n(z),
$$
because $y^i$ and $y^e$ lie on the line passing trough the points $z$ and $y$,
we deduce that both  points $y^i$ and $y^e$  {\color{black}belong either} to $\Ac_i(z,1)$
or to $\Ac_e(z,1)$. This contradicts the fact that $y^i\in \Ac_i(z,1)$ and $y^e\notin \Ac_i(z,1)${\color{black}; this contradiction proves}
 \eqref{yn0}.

By \eqref{divID} and the divergence theorem, 
\begin{align*}
H_s(z)&=\frac{1}{2s\,\omega_{n-2}} PV\!\!\! \int_{\Real^n}  \widehat\chi_{\Sc}(z,y) \text{div}_y\big[(z-y)|z-y|^{-(n+2s)}\big]\, dy\\
&=\frac{1}{s\,\omega_{n-2}} PV\!\!\! \int_{\partial\Ac_i(z,1)} |z-y|^{-(n+2s)} (z-y)\cdot \n_{\Ac_i}(y)   \, dy\\
&=\frac{1}{s\,\omega_{n-2}} PV\!\!\! \int_{\Sc} |z-y|^{-(n+2s)} (z-y)\cdot \n_{\Ac_i}(y)   \, dy,
\end{align*}
where the last equality follows from \eqref{yn0}.
\end{proof}

\begin{remark}
For {\color{black} a surface without boundary, the outward normal to $\Ac_i(z,1)$ coincides with the normal to the surface. For a surface $\Sc$ with boundary, at points $z\in \Sc$ such that all the half-lines starting at $z$ intersect $\Sc$ in at most one point (not counting $z$) the normal $\n_{\Ac_i}$ coincides with the normal $\n$ to $\Sc$. Indeed, for} these particular kinds of surfaces, \eqref{33} can be written as
$$
H_s(z)=\frac{1}{s\omega_{n-2}} PV\!\!\!\int_{\Sc} |z-y|^{-(n+2s)}(z-y)\cdot \bn(y)\,dy.
$$
\end{remark}
\vspace{5mm}
 \color{black}

{\color{black}Finally, s}imilar to what is done in \cite{AV} for surfaces {\color{black} without boundary}, we define the nonlocal directional curvature at $z$ in the direction of $\e\in \Tc_z(\Sc)$ by 
\beqn\label{Kse}
K_{s,\e}(z):={PV\!\!\!}\int_{\pi(z,\e)}|y'-z|^{n-2}\,\widehat\chi_\Sc(z,y)\,|z-y|^{-(n+2s)} dy,
\eeqn
where 
the notation is consistent with that used in \eqref{dirNMC}.  When the surface $\Sc$ is the boundary of an open set, this formula for the nonlocal directional curvature is consistent with \eqref{dirNMC}.  
The nonlocal mean and directional curvatures are related through the formula
$$H_s(z)=\frac{1}{\omega_{n-2}} \int_{\{\e\in\Tc_z(\Sc)\, |\ |\e|=1\}}K_{s,\e}(z)\, d\e,$$
meaning that the nonlocal mean curvature is the average of the nonlocal directional curvatures.
Unsurprisingly, just like the nonlocal curvatures for surfaces {\color{black} without boundary}, the nonlocal quantities converge to their local counterparts in the appropriate limit.

\begin{proposition}\label{limitNC}
Let $\Sc$ be  an oriented $C^1$ surface with orientation $\n$. 
For all $z\in\Sc$,
\begin{equation}\label{limitNLcurv}
\quad\lim_{s\rightarrow 1/2^-}(1-2s)K_{s,\e}(z)=K_\e(z)\quad \text{and}\quad\lim_{s\rightarrow 1/2^-}(1-2s)H_s(z)=H(z).
\end{equation}
\end{proposition}

\begin{proof}
Only a proof of \eqref{limitNLcurv}$_1$ will be given, because the proof of \eqref{limitNLcurv}$_2$ is similar.  Fix $\ve>0$ and notice that
\begin{align*}
\Big |\int _{\pi(z,\e)\backslash B_\ve(z)} (1-2s) \frac{|y'-z|^{n-2}}{|y-z|^{n+2s}} \widehat\chi_\Sc(z,y)\, dy \Big| &\leq \int _{\pi(z,\e)\backslash B_\ve(z)} \frac{1-2s}{|y-z|^{2+2s}} \, dy\\
& = \int _{\pi(z,\e)\backslash B_\ve(0)} \frac{1-2s}{|y|^{2+2s}} \, dy;
\end{align*}
the last integral goes to zero as $s$ goes to $1/2$.  Thus,
$$
\lim_{s\rightarrow 1/2^-} (1-2s) K_{s,\e}(z) = \lim_{s\rightarrow 1/2^-}  \int _{\pi(z,\e)\cap B_\ve(z)} (1-2s) \frac{|y'-z|^{n-2}}{|y-z|^{n+2s}} \widehat\chi_\Sc(z,y)\, dy,
$$
meaning that, for $s$ approaching $1/2$, $(1-2s)K_{s,\e}(z)$ only depends on the surface $\Sc$ in a small neighborhood of $z$.  
Choose $\ve$ small enough so that the part of $\Sc$ inside $B_\ve(z)$ is diffeomorphic to a disk.
It is possible to find an open set $E$ with smooth boundary such that $\Sc\cap B_\ve(z)=\partial E \cap B_\ve(z)$. We denote by  $K^{\partial E}_{s,\e}(z)$ the nonlocal directional curvature of $\partial E$ at $z$, so that
$$
\lim_{s\rightarrow 1/2^-}(1-2s)K_{s,\e}(z)=\lim_{s\rightarrow 1/2^-}(1-2s)K^{\partial E}_{s,\e}(z)=K^{\partial E}_\e(z),
$$
where \eqref{curvlim} has been utilized.  Since $K^{\partial E}_\e(z)=K_\e (z)$, we have, as desired, that
$$
\lim_{s\rightarrow 1/2^-}(1-2s)K_{s,\e}(z)=K_\e(z).
$$
\end{proof}

\noindent
{\bf Acknowledgement.}
{\color{black}We gratefully acknowledge the contributions of three reviewers, whose pertinent comments led us to a more complete and clearer exposition of our ideas.
}

\medskip
Received xxxx 20xx; revised xxxx 20xx.
\medskip


\begin{thebibliography}{99}

\bibitem{AV} N. Abatangelo and E. Valdinoci, A notion of nonlocal curvature, \emph{Numer. Funct. Anal. Optim.}, {\bf{35}} (2014), 793--815.
 
 
 
 
\bibitem{Alb} G. Alberti, Distributional Jacobian and singularities of Sobolev maps, \emph{Ric.\ Mat.}, \textbf{LIV} (2006), 375--394.

\bibitem{AFP} L. Ambrosio, N Fusco and D Pallara, \emph{Functions of Bounded Variation and Free Discontinuity Problems}, Oxford University Press, Oxford, 2000.


\bibitem{AdPM} L. Ambrosio, G. De Philippis and L. Martinazzi, $\Gamma$-convergence of nonlocal perimeter functionals, \emph{Manuscripta Math.} \textbf{134} (2011), 377--403.

\bibitem{Cetal} X. Cabr\'e, M.M. Fall, J. Sol\`a-Morales and T. Weth, Curves and surfaces with constant nonlocal mean curvature: meeting Alexandrov and Delauney, \emph{Journal f\"ur die reine und angewandte Mathematik}, published on line 2016-04-16.

\bibitem{LC} L. Caffarelli, Surfaces minimizing nonlocal energies, \emph{Rend. Lincei Mat. Appl.} {\bf{20}} (2009), 281--299.

\bibitem{LCMag} L. Caffarelli, The mathematical idea of diffusion, Enrico Magenes Lecture, Pavia, March 2013.

\bibitem{CRS} L. Caffarelli, J.-M. Roquejoffre, and O. Savin, Nonlocal minimal surfaces. \emph{Comm. Pure Appl. Math.}, \textbf{63} (2010), 1111--1144.

\bibitem{CS} L. Caffarelli and P. Souganidis, Convergence of nonlocal threshold dynamics approximations to front propagation, \emph{Arch. Rational Mech. Anal.} \textbf{195} (2010), 1--23.

\bibitem{CVa} L. Caffarelli and E. Valdinoci, Regularity properties of nonlocal minimal surfaces via limiting
arguments,  \emph{Adv. Math.} {\bf 248} (2013), 843--871.

\bibitem{CVb} L. Caffarelli and E. Valdinoci, Uniform estimates and limiting arguments for nonlocal minimal surfaces, \emph{Calc. Var. Partial Differential Equations}, \textbf{41} (2011), 203--240.

\bibitem{CMP1} A. Chambolle, M. Morini and M. Ponsiglione, A nonlocal mean curvature flow and its semi-implicit time-discrete approximation, \emph{SIAM J. Math. Anal.} \textbf{44} (2012), 4048--4077.

\bibitem{CMP2} A. Chambolle, M. Morini and M. Ponsiglione, Minimizing movements and level set approaches to nonlocal variational geometric flows, \emph{Geometric partial differential equations}, CRM Series, 15, Ed. Norm., Pisa, (2013), 93--104.

\bibitem{CMP3} A. Chambolle, M. Morini and M. Ponsiglione, Nonlocal Curvature Flows, \emph{Arch. Rational Mech. Anal.}, \textbf{218} (2015), 1263--1329.

\bibitem{DiVa} S. Dipierro and E. Valdinoci, Nonlocal minimal surfaces: Interior regularity, quantitative estimates and boundary stickiness, preprint, \arXiv{1607.06872v2}.

\bibitem{DiSaVa} S. Dipierro, O. Savin, and E. Valdinoci, Boundary behaviour of nonlocal minimal surfaces, \emph{J. Funct. Anal.}, {\bf 272} (2017), 1791--1851.

\bibitem{DoC} M.P. Do Carmo, \emph{Differential Geometry of Curves and Surfaces}, Prentice Hall, New Jersey, 1976.



\bibitem{FiVa} A. Figalli and E. Valdinoci, Regularity and Bernstein-type results for nonlocal minimal surfaces, \emph{J. Reine Angew. Math.}, published on line 2015-04-28.

\bibitem{EvGa}  L.C. Evans and R.F. Gariepy, \emph{Measure Theory and Fine Properties of Functions. Revised edition}. CRC Press, 2015.

\bibitem{FS} D. Frenkel and B. Smit, \emph{Understanding Molecular Simulations}, Academic Press, 2002.

\bibitem{Giusti} E. Giusti, \emph{Minimal surfaces and functions of bounded variation}. Birkh\"auser Boston, 1984.

\bibitem{Im}  C. Imbert, Level set approach for fractional mean curvature flows, \emph{Interfaces Free Bound.}, \textbf{11} (2009), 153--176.

\bibitem{MaVa} F. Maggi and E. Valdinoci, Capillarity problems with nonlocal surface tension energies, preprint, \arXiv{1606.08610}.


\bibitem{Merri} B. Merriman, J.K. Bence, and S. Osher, Diffusion Generated Motion by Mean Curvature, CAM report, Department of Mathematics, University of California, Los Angeles, 1992.

\bibitem{PPGL} P. Podio-Guidugli, A notion of nonlocal Gaussian curvature, \emph{Rend.\ Lincei: Mat.\ e Appl.}, {\bf 27} (2016), 181--193.

\bibitem{SV}  O. Savin and E. Valdinoci,  $\Gamma$-convergence for nonlocal phase transitions, \emph{Ann. Inst. H. Poincar\'e Anal. Non Lin\'eaire}, \textbf{29} (2012), 479--500.

\bibitem{SaVa}  O. Savin and E. Valdinoci, Regularity of nonlocal minimal cones in dimension 2, \emph{Calc. Var. Partial Differential Equations}, \textbf{48} (2013), 33--39.

\bibitem{Weyl} H. Weyl, On the volume of tubes, \emph{Am.\ J.\ Math.}, \textbf{61} (1939), 461--472.


%
%
%
%
%
%
%
%
%
%
%

\end{thebibliography}
\end{document}